\documentclass[12pt,a4paper]{amsart}
\usepackage{amsmath,amsfonts,amsthm,amsopn,color,amssymb,enumitem}
\usepackage[a4paper]{geometry}
\usepackage{palatino}
\usepackage{graphicx}
\usepackage[colorlinks=true]{hyperref}
\hypersetup{urlcolor=blue, citecolor=red, linkcolor=blue}
\allowdisplaybreaks

\usepackage{cite}
\usepackage{relsize}
\usepackage{esint}
\usepackage{verbatim}
\usepackage{mathrsfs}
\usepackage{xcolor}
\usepackage{xfrac}

\newcommand{\e}{\varepsilon}

\newcommand{\R}{\mathbb{R}}

\newcommand{\de}{\partial}
\newcommand{\weakto}{\rightharpoonup}

\renewcommand{\d }{\delta }

\newcommand{\dsh}{{2^\sharp}}
\newcommand{\dst}{{2^*}}
\newcommand{\hsp}{\hspace{0.2cm}}

\renewcommand{\(}{\left(}
\renewcommand{\)}{\right)} 

\newcommand{\W}{W}
\newcommand{\iM}{i^*_{\Omega}}
\newcommand{\idM}{i^*_{\partial \Omega}}

\newcommand{\Z}{\mathcal Z}
\newcommand*{\scal}[1]{\left\langle #1 \right\rangle}

\newcommand*{\abs}[1]{\left\vert #1\right\vert}
\newcommand*{\norm}[1]{\left\Vert #1\right\Vert}
\newcommand*{\bigo}[1]{\mathcal O\left( #1 \right)}

\newtheorem{theorem}{Theorem}[section]
\newtheorem*{theorem*}{Theorem}
\newtheorem{lemma}[theorem]{Lemma}

\newtheorem{proposition}[theorem]{Proposition}

\theoremstyle{definition}

\title[Blow-up solutions for a doubly critical elliptic problem]{The role of the boundary in the existence of blow-up solutions for a doubly critical elliptic problem}

\author{Sergio Cruz-Blázquez}
\address{Sergio Cruz Blázquez, Departamento de Análisis Matemático, Universidad de Granada, Avenida de Fuente Nueva s/n, 18071 Granada (Spain)}
\email{sergiocruz@ugr.es}
\thanks{S.C. acknowledges financial support from the Spanish Ministry of Universities and Next Generation EU funds, through a \textit{Margarita Salas} grant from the University of Granada, by the FEDER-MINECO Grant PID2021-122122NB-I00 and by J. Andalucia (FQM-116). He has also been partially supported by the IN$\delta$AM-GNAMPA project \textit{Fenomeni di blow-up per equazioni nonlineari}, code CUP\_E55F22000270001. This work was carried out during his long visit to the University “Sapienza Università di Roma”, to which he is grateful. He would also like to thank Prof. Angela Pistoia for many useful conversations about this problem. \newline \textbf{Declarations of interest:} none.}


\begin{document}
	\maketitle
	\begin{abstract}
		In this paper we consider a doubly critical nonlinear elliptic problem with Neumann boundary conditions. The existence of blow-up solutions for this problem is related to the blow-up analysis of the classical geometric problem of prescribing negative scalar curvature $K=-1$ on a domain of $\R^n$ and mean curvature $H=D(n(n-1))^{-1/2}$, for some constant $D>1$, on its boundary, via a conformal change of the metric. Assuming that $n\geq6$ and $D>\sqrt{(n+1)/(n-1)}$, we establish the existence of a positive solution which concentrates around an elliptic boundary point which is a nondegenerate critical point of the original mean curvature.
	\end{abstract}
	\section{Introduction}
	Let $(M,\de M,g)$ be a compact Riemannian manifold of dimension $n\geq 3$ with scalar curvature $S_g$ and boundary mean curvature $H_g$. Given two smooth functions $\tilde S$ and $\tilde H$, a classical geometric problem consists in asking whether for some conformal metric $\tilde g = u^{4\over n-2}g$, with $u$ a smooth and positive function, these functions can be achieved as the scalar and boundary mean curvatures of $M$ with respect to $\tilde g$, respectively. 
	
	\medskip
	
	From the analytical point of view, this problem is equivalent to finding a positive solution of the following elliptic equation, which has a double critical nonlinearity (see \cite{C}):
	\begin{equation}\label{sm}
		\left\lbrace\begin{array}{ll}
			\frac{-4(n-1)}{n-2}\Delta_g u + S_g u = \tilde S u^{n+2\over n-2} & \text{in}\hsp M, \\[1ex]
			\frac{2}{n-2}\frac{\de u}{\de \nu} + H_gu = \tilde Hu^{n\over n-2} & \text{on}\hsp \de M,
		\end{array}\right.
	\end{equation}
	where $\Delta_g$ stands for the Laplace-Beltrami operator associated to $g$ and $\nu$ is the unit outer normal to $\de M$.
	
	\medskip 
	The study of \eqref{P} was initiated with \cite{C}, wherein first criteria for the existence of solutions were established up to Lagrange multipliers. A particular case that has attracted considerable attention is the so-called \textit{Escobar problem}, consisting of prescribing constant curvature functions $\tilde S$ and $\tilde H$. Partial existence results depending on the dimension or the geometry of $M$ were given in \cite{A,E1,M1,M2,MN} for the case $\tilde S=0$, and in \cite{E2,BC} for the case of minimal boundary $\tilde H=0$. The complete problem with constant curvatures different from zero has only been treated when $\tilde S>0$, see e.g. \cite{E4,HY1, HY2, crs}.
	
	\medskip 
	
	Regarding the case of nonconstant functions, most of the results available concern the case in which one of the curvatures is identically zero: the works \cite{BEO1, BEO2,L} study the problem with $\tilde H=0$ on the half-sphere, while \cite{ACA,CXY,DMA2,XZ} considered the case $\tilde S=0$ on the unit ball of $\R^n$.
	
	\medskip 
	
	On the other hand, the case of variable $\tilde S$ and $\tilde H$ has been comparatively much less studied. We mention the works \cite{AYM,BCP}, in which the authors prescribe small perturbations of constant functions on the $n-$dimensional unit ball, \cite{CHS} for the case of negative curvatures on manifolds of negative conformal invariant, the problem on the three dimensional half sphere treated in \cite{DMA1} with $\tilde S>0$, and the case with $\tilde S<0$ on manifolds of nonpositive conformal invariant that has been recently studied in \cite{CMR}. The blow-up analysis performed in the latter work shows that the existence of blow-up solutions for the problem \eqref{sm} with $\tilde S<0$ is ruled by the \textit{scaling invariant} function $ D:\de M\to \R$ given by $$ D(p)= {\tilde H(p)\over \vert\tilde S(p)\vert^{1/2}},$$ with \textit{bubbling of solutions} possibly occurring only at points $p\in\de M$ with $ D(p) > (n(n-1))^{-1/2}$.
	
	\medskip
	
	Motivated by this analysis, in this paper we consider the following doubly critical elliptic problem posed in a smooth bounded domain $\Omega$ of $\R^n$:
	
	\begin{equation}\label{P}
		\left\lbrace\begin{array}{ll}
			\frac{-4(n-1)}{n-2}\Delta u + \mu u = -u^{n+2\over n-2} & \text{in}\hsp \Omega, \\[1ex]
			\frac{2}{n-2}\frac{\de u}{\de \nu} + H(x)u = \frac{D}{\sqrt{n(n-1)}}u^{n\over n-2} & \text{on}\hsp \de \Omega,
		\end{array}\right.
	\end{equation}
	where $H(x)$ denotes the mean curvature of $\de \Omega$, $\nu$ its unit normal vector pointing outwards, $\mu$ is a positive and large parameter and $D> 1$. 
	
	\medskip 
	
	Our goal is to construct solutions of \eqref{P} that blow-up around boundary points as $\mu\to +\infty.$ We prove the following result:
	\begin{theorem}\label{main-theorem} Let $\Omega$ be a smooth bounded domain of $\R^n$, and let $H$ denote the mean curvature of $\de \Omega$. Suppose that $n\geq 6$ and there exists an elliptic point $p\in\partial \Omega$ which is a nondegenerate critical point of $H$. Then, for every $D> \sqrt{(n+1)/(n-1)}$, the problem \eqref{P} admits a solution that concentrates around $p$ when $\mu \to +\infty$.
	\end{theorem}
	In fact, this is a counterpart of Theorem 1.1 (b) in \cite{Rey} (see also \cite{A-M}), which was dealing with the case of positive interior curvature and homogeneous Neumann boundary conditions, that is,
	\begin{equation}\label{p-rey}
		\left\lbrace\begin{array}{ll}
			\frac{-4(n-1)}{n-2}\Delta u + \mu u = u^{n+2\over n-2} & \text{in}\hsp \Omega, \\[1ex]
			\frac{\de u}{\de \nu} = 0 & \text{on}\hsp \de \Omega.
		\end{array}\right.
	\end{equation}
	In the aforementioned work, the author constructs solutions of \eqref{p-rey} in dimension $n\geq 6$ that concentrate around boundary points $p$ that are nondegenerate critical points of $H$ and satisfy $H(p)>0$. In comparison, our result imposes an additional condition on the geometry of $\de\Omega$ near $p$, which is equivalent to saying that $\bar \Omega$ is locally on one side of the tangent plane at $p$. This condition becomes imperative for defining our approximating solution. As opposed to what happens in \cite{Rey}, the solutions of the limiting equation in the entire space $\R^n$ are not suitable to build an effective ansatz in our case. This limitation arises from our nonlinear boundary condition, and prompts the use of solutions of a problem in the half-space, as detailed in Section \ref{sec-setting}. 
	
	\medskip
	
	Once we have established our ansatz, we argue following a classical Ljapunov-Schmidt reduction. This then leads us to find a critical point for the \textit{reduced energy} $\mathfrak E:(0,+\infty)\times \R^{n-1}\to \R$ associated to the Euler-Lagrange functional of \eqref{P}, which has the following expansion after a suitable choice of the parameters:
	\begin{equation*}
		\mathfrak E(d,\xi) = \mathbb{E}+\frac1\mu\(\mathscr C_n(D)H(\xi) d + \frac{A^2}{2}d^2\) + \bigo{\frac1{\mu^2}},
	\end{equation*}
	where $\mathbb{E}$ and $A$ are constants and $\mathcal C_n:(1,+\infty)\to\R$ is a continuous function satisfying 
	\begin{equation*}
		\mathcal C_n(D) \:\left\lbrace\begin{array}{ll}
			>0 & \text{if }\: 1<D<\sqrt{n+1\over n-1}, \\[0.2cm]
			=0 & \text{if }\: D=\sqrt{n+1\over n-1}, \\[0.2cm]
			<0 & \text{if }\: D>\sqrt{n+1\over n-1}.
		\end{array} \right.
	\end{equation*}
	for all $n\geq 5$. In view of the above expansion, it is natural to conjecture that the existence of solutions may be recovered in the case $H(p)<0$ by choosing $1 < D < \sqrt{(n+1)/(n-1)}$, even if the approach presented here can not be applied. We believe this can be accomplished by placing our approximating solution on the boundary of $\Omega$ using the Fermi coordinates, in the spirit of \cite{cpv,CV}. On the other hand, the case with $D=\sqrt{(n+1)/(n-1)}$ poses a greater level of complexity, requiring the computation of the expansion to higher orders and the refinement of the ansatz. Both subjects will be the focus of future research.
	
	\medskip
	
	Finally, it is worth noticing that the extension of these results to the case $n=5$ needs techniques that are specific of that dimension, as illustrated in \cite{ZQW} for the problem \eqref{p-rey}.
	
	\medskip
	
	The rest of the paper is organized as follows. Section \ref{sec-prelim} is devoted to notation and preliminaries. In section \ref{sec-setting} we find an adequate ansatz and rewrite \eqref{P} as a system of equations, in a standard Ljapunov-Schmidt fashion. For the reader's convenience, we divide the resolution of these equations into sections \ref{sec-aux} and \ref{sec-ls}, obtaining a clean proof of Theorem \ref{main-theorem} in section \ref{sec-proof}. In the appendices \ref{app-invertibility} and \ref{app-contraction} we collect the most technical results of this procedure.
	\section{Preliminaries}\label{sec-prelim}
	\subsection{Notation} Throughout this paper, we will distinguish the last coordinate of $\R^n$ from the rest, by means of the notation $x=(\bar x,x_n)$, with $x\in \R^n$, $\bar x\in \R^{n-1}$ and $x_n\in \R$. $B'(\xi,\rho)$ will be used to denote the $n-1$ dimensional Euclidean ball in $\bar x$ centered at $\xi \in \R^{n-1}$ of radius $\rho>0$. Moreover, we denote by $C(\xi,\rho)$ the cylinder 
	\begin{equation*}
		C(\xi,\rho) = \{x\in \R^n:\:\bar x\in B'(\xi,\rho)\wedge 0<x_n<\rho\},
	\end{equation*} 
	and $A(\xi,\rho)=C(\xi,\rho)\backslash C\left(\xi,\frac\rho 2\right)$.  To lighten the notation, we often omit the center when this is the origin of coordinates. 
	
	\smallskip 
	
	Given $\Omega\subset \R^n$ and $\delta >0$, the notation ${\Omega\over \delta}$ represents the subset
	
	\begin{equation*}
	{\Omega\over\delta} = \left\lbrace \frac{x}{\delta}:x\in \Omega \right\rbrace.
	\end{equation*}
	
	\smallskip 
	
	For $D>1$ and $0\leq m< n+1$ we set $$\beta^m_n(D)=\omega_{n-2}\int_0^{+\infty}\frac{r^{n-2+m}}{(r^2+D^2-1)^n}dr = \omega_{n-2}\frac{\Gamma\(\frac{n-m+1}{2}\)\Gamma\(\frac{n+m-1}{2}\)}{(D^2-1)^{n-m+1\over 2}2\Gamma(n)}.$$
	
	For notational convenience, sometimes we omit volume or surface elements in integrals: we will also denote by $C$ positive constants that may vary from line to line, or also within the same one. 
	
	\subsection{Setting of the problem} 
	Let $p\in\de \Omega$ be an elliptic point, that is, a boundary point at which all the principal curvatures are positive. We know that in a small neighborhood of $p$ the domain $\Omega$ lies on one side of the tangent plane. Up to rotating and translating $\Omega$, we can assume that $p=0$ and there exist $\rho>0$ and a smooth function $\varphi:B'(0,\rho)\to \R^+_0$ such that $\varphi(0) =0$, $\nabla\varphi(0)=0$ and
	\begin{equation*}
		\begin{split}
			\Omega\cap C(0,\rho)&=\{(\bar x,x_n)\in C(0,\rho): x_n\geq \varphi(\bar x)\}.
		\end{split}
	\end{equation*}
	Therefore, without loss of generality, we have the following expansion for $\varphi$ around $0$: 
	\begin{equation}\label{expansion}
		\varphi(\bar x) = \sum_{i=1}^{n-1}k_i{x_i}^2 + \bigo{\abs{\bar x}^3}, \quad \bar x \in B'(0,\rho),
	\end{equation}
	with $k_i> 0$ for every $i=1,\ldots,n-1.$ Notice that, with this notation,
	\begin{equation*}
		H(0)=\frac{2}{n-1}\sum_{i=1}^{n-1}k_i.
	\end{equation*}
	We also define the interspace set $\Sigma=C(0,\rho)\setminus \Omega$, which we can write as
	\begin{equation}\label{sigma-def}
		\Sigma = \{(\bar x,x_n):\:\bar x\in B'(0,\rho)\wedge 0<x_n<\varphi(\bar x)\}.
	\end{equation}
	
	Finally, let us introduce the energy functional associated to \eqref{P}, defined for every $u\in H^1(\Omega)$:
	\begin{equation}\label{Energia}
		\begin{split}
			E(u) &= \frac{2(n-1)}{n-2}\int_\Omega \abs{\nabla u}^2 + \frac{\mu}{2}\int_\Omega u^2 + \frac{n-2}{2n}\int_\Omega \abs{u}^\dst \\&-\frac{(n-2)D}{\sqrt{n(n-1)}}\int_{\de\Omega}\abs{u}^\dsh +(n-1)\int_{\de\Omega}Hu^2
		\end{split}
	\end{equation}
	Clearly, critical points of \eqref{Energia} provide solutions of \eqref{P}.
	
	\section{The Linear Theory}\label{sec-setting}
	Since $\mu>0$, we can endose $H^1(\Omega)$ the following norm, which is equivalent to the standard one:
	\begin{equation*}
		\norm{u}_\mu = \frac{4(n-1)}{n-2}\int_\Omega \abs{\nabla u}^2 + \mu \int_\Omega u^2.
	\end{equation*}
	By the well-known Sobolev embedding theorem and trace inequality, we have the following continuous embedding maps:
	\begin{equation*}
		\begin{split}
			i_\Omega:H^1(\Omega)&\to L^\dst(\Omega),\hsp \text{with}\hsp \dst = \frac{2n}{n-2}, \\
			i_{\de\Omega}:H^1(\Omega)&\to L^\dsh(\de \Omega),\hsp \text{with}\hsp \dsh = \frac{2(n-1)}{n-2}.
		\end{split}
	\end{equation*}
	Let $i^*_\Omega$ and $i^*_{\de \Omega}$ denote the adjoint operators. Then, by definition, given $\mathfrak f\in L^{2n\over n+2}(\Omega)$, $i^*_\Omega(\mathfrak f)$ is the unique solution in $H^1(\Omega)$ to the boundary value problem
	\begin{equation*}
		\left\lbrace\begin{array}{ll}
			\frac{-4(n-1)}{n-2}\Delta u +\mu u = \mathfrak f & \text{in}\hsp \Omega, \\[1ex]
			\frac{\de u}{\de \nu} = 0 & \text{on}\hsp \de \Omega.
		\end{array}\right.
	\end{equation*}
	Analogously, if $\mathfrak g\in L^{2(n-1)\over n}(\de \Omega)$, $i^*_{\de \Omega}(\mathfrak g)$ denotes the solution of 
	\begin{equation*}
		\left\lbrace\begin{array}{ll}
			\frac{-4(n-1)}{n-2}\Delta u +\mu u = 0 & \text{in}\hsp \Omega, \\[1ex]
			\frac{\de u}{\de \nu} = \mathfrak g & \text{on}\hsp \de \Omega.
		\end{array}\right.
	\end{equation*}
	Therefore, $u$ solves \eqref{P} if and only if 
	\begin{equation}\label{P-2}
		u-i^*_\Omega\(-u^{n+2\over n-2}\)-i^*_{\de \Omega}\(\frac{n-2}{2}\(\frac{D}{\sqrt{n(n-1)}}u^{n\over n-2}-H(x)u\)\) = 0.
	\end{equation}
	\subsection{The Ansatz} We will construct a solution of \eqref{P-2} as $\mu\to +\infty$ that looks like a \textit{bubble} centered at the origin. More precisely, we define for $x=(\bar x, x_n)\in \R^{n-1}\times \R^+$:
	\begin{equation}\label{bubble}
		U_{\delta,\xi}(x)= \frac{\alpha_n \delta^{n-2\over 2}}{(\abs{\bar x-\xi}^2+(x_n+\delta D)^2-\delta^2)^{n-2\over 2}},
	\end{equation}
	where $\xi\in \R^{n-1}$ and $\delta>0$ are parameters and $\alpha_n=(4n(n-1))^{n-2\over 4}$. When $D>1$, this $n-$dimensional family represents all the solutions of the problem in $\R^n_+$ (see \cite{cfs}):
	\begin{equation}\label{L}
		\left\lbrace\begin{array}{ll}
			\frac{-4(n-1)}{n-2}\Delta u = -u^{n+2\over n-2} & \text{in}\hsp \R^n_+, \\[1ex]
			\frac{2}{n-2}\frac{\de u}{\de \nu} = \frac{D}{\sqrt{n(n-1)}}u^{n\over n-2} & \text{on}\hsp \de \R^n_+.
		\end{array}\right.
	\end{equation}
	We highlight the following scaling property of the family of functions $U_{\delta,\xi}$, which will be used continuously for our computations:
	\begin{equation*}
		U_{\delta,\xi}(\delta y) = \frac{1}{\delta^{n-2\over 2}}U_{1,\delta^{-1}\xi}(y), \quad \text{for every } y\in\R^n_+.
	\end{equation*}
	We set
	\begin{equation}\label{A}
		W_{\mu,\xi}(x)= \chi(x)U_{\delta,\xi}(x),
	\end{equation}
	where $\delta(\mu)\to 0^+$ and $\chi$ is a cut-off function with support on $C(0,\rho)$. Without loss of generality we can assume that $\chi = 1$ in $C\left(0,{\rho\over 2}\right)$ and
	\begin{equation}\label{cond-cutoff}
		\abs{\nabla \chi} \leq \frac{C}{\abs{x}}, \quad \abs{\Delta\chi} \leq \frac{C}{\abs{x}^2} \quad \text{on }\:  A(0,\rho).
	\end{equation}
	
	Moreover, we define
	\begin{equation*}
		\mathcal W_{\mu,\xi} = W_{\mu,\xi} + \Phi_{\mu,\xi},
	\end{equation*}
	where $\Phi$ is chosen in $\mathcal K^\perp$, defined as follows: let $\{\mathfrak J_j:j=1,\ldots,n\}$ be the functions generating the space of solutions of the linearized problem
	\begin{equation*}
		\left\lbrace\begin{array}{ll}
			\frac{-4(n-1)}{n-2}\Delta v + \frac{n+2}{n-2}U^{4\over n-2}v = 0 & \text{in}\hsp \R^n_+, \\[1ex]
			\frac{2}{n-2}\frac{\de v}{\de \nu} - \frac{D}{n-2}\sqrt{n\over n-1}U^{2\over n-2}v=0 & \text{on}\hsp \de \R^n_+,
		\end{array}\right.
	\end{equation*}
	which are given by the formulas
	\begin{equation*}\label{Ji}
		\mathfrak J_i(x)=\frac{\partial U_{\delta,\xi}}{\partial x_i}\bigg\vert_{\substack{\delta=1\\ \xi=0}}(x)=\frac{\alpha_n(2-n) x_i}{\left(|\tilde x|^2+(x_n+\mathfrak D_n(p))^2-1\right)^{\frac{n}{2}}},
	\end{equation*}
	for $i=1, \ldots, n-1$, and
	\begin{equation*}\label{Jn}\begin{aligned}
			\mathfrak J_n(x)=\frac{\partial U_{\delta,\xi}}{\partial \delta}\bigg\vert_{\substack{\delta=1\\ \xi =0}}(x)
			=\frac{\alpha_n(2-n)}{2}\frac{|x|^2+1-\mathfrak D_n(p)^2}{\left(|\tilde x|^2+(x_n+\mathfrak D_n(p))^2-1\right)^{\frac{n}{2}}},\end{aligned}
	\end{equation*}
	(see \cite[Th. 2.1]{cpv}), and define 
	\begin{equation*}
		\mathcal Z_j(x) = \frac{1}{\delta^{n-2\over 2}}\mathfrak J_j\(\frac{x-\xi}{\delta}\)\chi(x).
	\end{equation*}
	Finally, let $\mathcal K$ denote the vector space $\mathcal K = \text{span}\,\(\mathcal Z_j:j=1,\ldots,n\)$, and let $\mathcal K^\perp$ be its orthogonal space with respect to the scalar product that induces the norm $\norm{\cdot}_\mu$, that is,
	\begin{equation*}
		\mathcal{K}^\perp = \left\lbrace v\in H^1(\Omega):\,\frac{4(n-1)}{n-2}\int_\Omega \nabla v\nabla\cdot\mathcal Z_j+\mu \int_\Omega v\,\mathcal Z_j = 0,\hsp\forall j=1,\ldots,n \right\rbrace.
	\end{equation*}
	
	\smallskip
	
	Let us denote by $\Pi$ and $\Pi^\perp$ the projections of $H^1(\Omega)$ to $\mathcal K$ and $\mathcal K^\perp$, respectively. Then, \eqref{P-2} results equivalent to solve the system of equations
	\begin{align}
		\Pi\left(u-i^*_\Omega\(-u^{n+2\over n-2}\)-i^*_{\de \Omega}\(\frac{n-2}{2}\(\frac{D}{\sqrt{n(n-1)}}u^{n\over n-2}-H(x)u\)\)\right) &= 0\label{bif}, \\[6px] \Pi^\perp\left(u-i^*_\Omega\(-u^{n+2\over n-2}\)-i^*_{\de \Omega}\(\frac{n-2}{2}\(\frac{D}{\sqrt{n(n-1)}}u^{n\over n-2}-H(x)u\)\)\right) &= 0 \label{aux}.
	\end{align} 
	
	\medskip
	
		In all the notations above, we omit the dependence of $\xi$ when $\xi = 0.$
		
	\section{Solving the Auxiliar Equation}\label{sec-aux} First, we find a choice of $\Phi_{\mu,\xi}$ so that $\mathcal W_{\mu,\xi}$ solves \eqref{aux}. This is achieved by means of a classical fixed point argument, so the most standard proofs will be postponed to the appendix for the sake of brevity. To simplify the notation, we sometimes omit the dependence on $\xi$, $\mu$ or both.
	
	\medskip
	
	Let us rewrite the equation \eqref{aux} as 
	\begin{equation*}
		\mathcal L(\Phi)+\mathcal N(\Phi)+\mathcal E =0,
	\end{equation*}
	where $\mathcal L$ is the linear operator
	\begin{equation}\label{linear-operator}
		\begin{split}
			\mathcal L(\varphi) = \Pi^\perp\(\Phi -i^*_\Omega\(-\frac{n+2}{n-2}W^{4\over n-2}\Phi\)-i^*_{\de\Omega}\(\frac{n}{2}\frac{D}{\sqrt{n(n-1)}}W^{2\over n-2}\Phi-\frac{n-2}{2}H\:\Phi\)\),
		\end{split}
	\end{equation}
	$\mathcal N$ is a nonlinear term given by 
	\begin{equation}\label{nonlinear-term}
		\begin{split}
			\mathcal N(\varphi) &= \Pi^\perp\Bigg(-i^*_\Omega\(-(W+\Phi)^{n+2\over n-2}+W^{n+2\over n-2}+\frac{n+2}{n-2}W^{4\over n-2}\Phi\)\\&-i^*_{\de\Omega}\bigg(\frac{D}{\sqrt{n(n-1)}}\(\frac{n-2}{2}(W+\Phi)^{n\over n-2}-\frac{n-2}{2}W^{n\over n-2}-\frac{n}{2}W^{2\over n-2}\Phi\)\bigg)\Bigg),
		\end{split}
	\end{equation}
	and $\mathcal E$ is the error, defined as
	\begin{equation}\label{error}
		\mathcal E = \Pi^\perp\( W-i^*_\Omega\(- W^{n+2\over n-2}\)-i^*_{\de \Omega}\(\frac{n-2}{2}\(\frac{D}{\sqrt{n(n-1)}} W ^{n\over n-2}-H\, W\)\)\).
	\end{equation}
	The main result of this section is the following:
	\begin{proposition}\label{solve-aux} For any compact subset $\mathfrak K \subset (0,+\infty)\times \R^{n-1}$ there exists $\mu_0>0$ such that for any $\mu>\mu_0$ and any $(d,\xi)\in\mathfrak K$ there exists a unique function $\Phi_\mu\in\mathcal K^\perp$ that solves \eqref{aux}. Moreover, the map $(d,\xi)\to \Phi_\mu$ is of class $C^1$ and 
		\begin{equation*}
			\norm{\Phi_\mu}_\mu \leq 
			\begin{cases}
				\displaystyle
				C\(\mu\delta^2+\delta\) & \text{if}\quad n\geq 7, \\[6pt]
				C\(\mu\delta^2\abs{\log\delta}^{2\over3}+\delta\) & \text{if}\quad n= 6.
			\end{cases}
		\end{equation*}
	\end{proposition}
	The most important steps of its proof are collected in the following lemmas. First, we show that the linear operator $\mathcal L$ defined in \eqref{linear-operator} is invertible in $\mathcal K^\perp$. 
	
	\begin{lemma}\label{invertibility}
		For any compact subset $\mathfrak K\subset (0,+\infty)\times \R^{n-1}$, there exist positive constants $C>0$ and $\e_0>0$ such that, for every $(d,\xi)\in\mathfrak K$, it holds
		\begin{equation*}
			\norm{\mathcal L(\phi)}_\mu\geq C\norm{\phi}_\mu \quad \forall \phi \in \mathcal K^\perp.
		\end{equation*}
	\end{lemma}
	In establishing this result, we follow the ideas presented in \cite[Lemma 8]{GMP18}, making the necessary adjustments to tailor them to our specific problem. The detailed proof is provided in Appendix \ref{app-invertibility}.
	
	\medskip
	
	The next lemma shows that the nonlinear term $\mathcal N$ acts as a contraction when restricted to a suitable ball of $H^1(\Omega)$.
	\begin{lemma}\label{contraction} There exists a small $r>0$ such that the nonlinear operator $N$ given by \eqref{nonlinear-term} is a contraction on $B_r(0)\subset H^1(M)$, that is to say, there exists a constant $0<\gamma<1$ such that
		\begin{equation*}
			\norm{N(\phi_1)-N(\phi_2)}_\mu \leq C \norm{\phi_1-\phi_2}_\mu 
		\end{equation*}
		for every $\phi_i\in H^1(M)$ with $\norm{\phi_i}_\mu\leq r$, $i=1,2.$
	\end{lemma}
	The proof of this technical lemma is analogous to that of \cite[Remark 10]{GMP18}, so we postpone it to Appendix \ref{app-contraction}.
	
	\medskip
	
	Finally, we need to estimate the size of the error term $\mathcal E$ defined in \eqref{error}. Here again we follow the methodology of \cite[Lemma 9]{GMP18}, estimating each integral term in the spirit of \cite{Rey}. To obtain accurate enough estimates, we split $\Omega\cap C(0,\rho)$ as $C(0,p)\setminus \Sigma$, where $\Sigma$ is the set defined in \eqref{sigma-def}. As can be seen throughout the paper, the mean curvature of $\de\Omega$ appears naturally when studying the integral terms in $\Sigma$.
	
	\begin{lemma}\label{error-size} Let $n\geq 6$. For any compact subset $\mathfrak K \subset (0,+\infty)\times \R^{n-1}$ there exists $\mu_0>0$ such that for any $\mu>\mu_0$ and any $(d,\xi)\in\mathfrak K$ it holds
		\begin{equation*}
			\norm{\mathcal E}_\mu \leq 
			\begin{cases}
				\displaystyle
				C\(\mu\delta^2+\delta\) & \text{if}\quad n\geq 7, \\[6pt]
				C\(\mu\delta^2\abs{\log\delta}^{2\over3}+\delta\) & \text{if}\quad n= 6.
			\end{cases}
		\end{equation*}
	\end{lemma}
	
	\begin{proof}
		Let us denote $\gamma_\Omega = i^*_\Omega\(-W^{n+2\over n-2}\)$ and $\gamma_{\de \Omega}=i^*_{\de \Omega}\(\frac{n-2}{2}\(\frac{D}{\sqrt{n(n-1)}} W ^{n\over n-2}-H\, W\)\)$. Integrating by parts, it is easy to see that
		\begin{align*}
			&\norm{\mathcal E}_\mu^2 = c_n\int_\Omega \abs{\nabla\(W-\gamma_\Omega-\gamma_{\de \Omega}\)}^2+\mu\int_{\Omega}\(W-\gamma_\Omega-\gamma_{\de \Omega}\)^2 \\ &= \int_{\Omega} c_n\((-\Delta)\(W-\gamma_\Omega - \gamma_{\de \Omega}\)+\mu(W-\gamma_\Omega-\gamma_{\de \Omega})\)\mathcal E +c_n\int_{\de\Omega}\frac{\de\(W-\gamma_\Omega-\gamma_\Omega\)}{\de \eta}\mathcal E
			\\ &= \int_{\Omega} \(\frac{-4(n-1)}{n-2}\Delta W+W^{n+2\over n-2}\)\mathcal E + \mu \int_{\Omega} W\,\mathcal E \\ &+2(n-1)\int_{\de \Omega}\(\frac{2}{n-2}\frac{\de W}{\de \eta}-\frac{D}{\sqrt{n(n-1)}}W^{n\over n-2}\)\mathcal E +2(n-1)\int_{\de \Omega}H W \mathcal E.
		\end{align*}
		Let us call $I_\Omega^1$, $I_\Omega^2$, $I_{\de \Omega}^1$ and $I_{\de \Omega}^2$ the integral terms from above, maintaining the order. By means of Hölder's inequality:
		\begin{align*}
			I_\Omega^1 \leq C \norm{\mathcal E}_\mu \(\int_\Omega\(-c_n\Delta W+W^{n+2\over n-2}\)^{2n\over n+2}\)^{n+2\over 2n} = C \norm{\mathcal E}_\mu \norm{-c_n\Delta W+W^{n+2\over n-2}}_{L^{2n\over n+2}(\Omega)}
		\end{align*}
		Observe that
		\begin{align*}
			&\norm{-c_n\Delta W + W^{n+2\over n-2}}^{2n\over n+2}_{L^{2n\over n+2}(\Omega)} = \int_{\Omega}\abs{-c_n\Delta W+W^{n+2\over n-2}}^{2n\over n+2}dx \\ &= \int_{A(\rho)\setminus \Sigma}\abs{-c_n \Delta\chi\, U_\delta -2c_n \scal{\nabla \chi,\nabla U_\delta}+U_\delta^{n+2\over n-2}\left(\chi^{n+2\over n-2}-\chi\right) }^{2n\over n+2} dx \\ &\leq C\int_{A({\rho\over \delta})\setminus {\Sigma\over\delta}}\abs{\frac{\abs{\Delta\chi(\delta y)} U_1(y)}{\delta^{n-2\over 2}} +\frac{\abs{\nabla \chi(\delta y)}\abs{\nabla U_1(y)}}{\delta^{n\over 2}}+\frac{U_1(y)^{n+2\over n-2}\left(\chi(y)-\chi(y)^{n+2\over n-2}\right)}{\delta^{n+2\over 2}} }^{2n\over n+2} \delta^n dy \\ &\leq C\int_{A({\rho\over \delta})} \abs{\frac{U_1(y)}{\abs{y}^2}+\frac{\abs{\nabla U_1(y)}}{\abs{y}}+\delta U_1(y)^{n+2\over n-2}}^{2n\over n+2} dy\leq C \delta^{n(n-2)\over n+2}.
		\end{align*}
			To obtain the last line, we have estimated the terms involving the derivatives of $\chi$ using \eqref{cond-cutoff}. We conclude that
			\begin{equation*}
			I^1_\Omega \leq C \norm{\mathcal E}_\mu \delta^{n-2\over 2}.
			\end{equation*}
		Similarly,
		\begin{align*}
			I^2_{\Omega} \leq C \mu \norm{\mathcal E}_\mu \norm{W}_{L^{2n\over n+2}(\Omega)},
		\end{align*}
		with
		\begin{align*}
			&\norm{W}_{L^{2n\over n+2}(\Omega)} \leq \delta^2 \(\int_{C\(0,\frac \rho \delta \)}U_1(y)^{2n\over n+2}dy\)^{n+2\over 2n} 
	\leq 	\begin{cases}
				\displaystyle
				C \delta^2 & \text{if}\quad n\geq 7, \\
				C \delta^2\abs{\log\delta}^{2\over 3} & \text{if}\quad n=6.
			\end{cases}
		\end{align*}
		As for the boundary terms, first we have
		\begin{align*}
			I^1_{\de \Omega} &\leq C \norm{\mathcal E}_\mu \(\int_{\de \Omega}\(\frac{2}{n-2}\frac{\de W}{\de \eta}-\frac{D}{\sqrt{n(n-1)}}W^{n\over n-2}\)^{2(n-1)\over n}\)^{n\over 2(n-1)} \\ &= C\norm{\mathcal E}_\mu \norm{\frac{2}{n-2}\frac{\de W}{\de \eta}-\frac{D}{\sqrt{n(n-1)}}W^{n\over n-2}}_{L^{2(n-1)\over n}(\de \Omega)}.
		\end{align*}
		We reason as follows:
		\begin{align*}
			&\norm{\frac{2}{n-2}\frac{\de W}{\de \eta}-\frac{D}{\sqrt{n(n-1)}}W^{n\over n-2}}_{L^{2(n-1)\over n}(\de \Omega)}^{2(n-1)\over n} \\ &= \int_{\de \Omega} \left(\frac{2}{n-2}\frac{\de \chi}{\de \eta}U_\delta+\frac{D}{\sqrt{n(n-1)}}U_\delta^{n\over n-2}\(\chi - \chi^{n\over n-2} \)\right)^{2(n-1)\over n}ds
			\\&\leq C \int_{\de \Omega \cap A\({\rho\over \delta}\)} \(\frac{U_1(y)}{\abs{y}}+U_1(y)^{n\over n-2}\(\chi(\delta y)-\chi(\delta y)^{n\over n-2}\)\)^{2(n-1)\over n}ds
\\ &\leq C \int_{A'\({\rho\over\delta}\)}\left(\frac{\abs{\(\bar y,\frac{\varphi(\delta \bar y)}{\delta}\)}^{-1}}{\(\abs{\bar y}^2+\({\varphi(\delta\bar y)\over \delta}+D\)^2-1\)^{n-2\over 2}}+\frac{1}{\(\abs{\bar y}^2+\({\varphi(\delta\bar y)\over \delta}+D\)^2-1\)^{n\over 2}}\right)^{2(n-1)\over n}d\bar y.
			\end{align*}
By Taylor expansion when $\delta\to 0,$ we have
\begin{equation}\label{exp1}
\begin{split}
&\frac{1}{\(\abs{\bar y}^2+\(\frac{\varphi(\delta\bar y)}{\delta}+D\)^2-1\)^{n\over 2}}\\&=\frac{1}{\(\abs{\bar y}^2+D^2-1\)^{n\over 2}} -\delta \frac{n D \sum k_i y_i^2}{\(\abs{\bar y}^2+D^2-1\)^{n+2\over 2}}+\bigo{\d^2},
\end{split}
\end{equation}
and
\begin{equation}\label{exp2}
	\begin{split}
	&\frac{\abs{\(\bar y,\frac{\varphi(\delta \bar y)}{\delta}\)}^{-1}}{\(\abs{\bar y}^2+\({\varphi(\delta\bar y)\over \delta}+D\)^2-1\)^{n-2\over 2}} \\ &= \frac{\abs{\bar y}^{-1}}{\(\abs{\bar y}^2+D^2-1\)^{n-2\over 2}}-\delta \frac{(n-2)D\abs{\bar y}^{-1}\sum k_i y_i^2}{(\abs{\bar y}^2+D^2-1)^{n\over 2}}+\bigo{\delta^2}.
	\end{split}
\end{equation}
Using \eqref{exp1}, \eqref{exp2} and the symmetry of the integral terms with respect to interchanging the variables $y_i$, for $i=1,\ldots,n-1$, we obtain
	\begin{align*}
	&\norm{\frac{2}{n-2}\frac{\de W}{\de \eta}-\frac{D}{\sqrt{n(n-1)}}W^{n\over n-2}}_{L^{2(n-1)\over n}(\de \Omega)}^{2(n-1)\over n} \\ &\leq C\int_{A'\({\rho\over\delta}\)} \left( \frac{1}{\(\abs{\bar y}^2+D^2-1\)^{n\over 2}}+ \frac{\abs{\bar y}^{-1}}{\(\abs{\bar y}^2+D^2-1\)^{n-2\over 2}}\right. 
	\\ &\left.\frac{\delta \abs{\bar y}^2}{\(\abs{\bar y}^2+D^2-1\)^{n+2\over 2}}+\delta \frac{\abs{\bar y}}{(\abs{\bar y}^2+D^2-1)^{n\over 2}}\right)^{2(n-1)\over n}d\bar y \leq C \delta^{n(n-2)\over n+2}.
	\end{align*}
	Hence, 
	\begin{equation*}
		I^1_{\de \Omega} \leq C \norm{\mathcal E}_\mu \delta^{n-2\over 2}.
	\end{equation*}


		An analogous reasoning yields:
		\begin{align*}
			&\(\int_{\de \Omega}(HW)^{2(n-1)\over n}\)^{n\over 2(n-1)} = \(\int_{B'(\rho)}H(x)\frac{\alpha_n^{2(n-1)\over n}\delta^{n-2\over 2}d\bar x}{\(\abs{\bar x}^2+(\varphi(\bar x)+\delta D)^2-\delta^2\)^{(n-1)(n-2)\over n}}\)^{n\over 2(n-1)} d\bar x\\ &= \delta\left[ H(0)\alpha_n^{2(n-1)\over n}\int_{B'\(\rho\over \delta\)}\frac{d\bar y}{\(\abs{\bar y}^2+D^2-1\)^{(n-1)(n-2)\over n}}\right. \\ &\left. -\frac{(n-1)(n-2)}{n}DH(0)^2\alpha^{2(n-1)\over n}\delta\int_{B'\(\rho\over\delta\)}\frac{\abs{\bar y}^2d\bar y}{\(\abs{\bar y}^2+D^2-1\)^{n-2+\frac{2}{n}}}\right]^{n\over 2(n-1)}d\bar y \leq C \delta 
		\end{align*}
	\end{proof}
	We are now able to complete the proof of Proposition \ref{solve-aux}. 
	\begin{proof}[\bfseries Proof of Proposition \ref{solve-aux}] By Lemmas \ref{invertibility} and \ref{error-size} there exist $c>0$ and $0<\gamma<1$ such that 
		\begin{equation*}
			\|\mathcal L^{-1}(\mathcal N(\phi)+\mathcal E)\|_\mu\leq c\left( \gamma\|\phi\|_\mu+\|\mathcal E\|_\mu\right). 
		\end{equation*}
		Now, if $\|\phi\|\leq 2c\|\mathcal E\|$, then the map $$\mathcal T(\phi):=\mathcal L^{-1}(\mathcal N(\phi)+\mathcal E)$$ is a contraction from the ball $\|\phi\|_\mu\leq 2c\|\mathcal E\|_\mu$ into itself by Lemma \ref{contraction}. 
		
		\medskip
		
		Finally, the fixed point theorem guarantees the existence of a unique solution of \eqref{aux}, $\Phi_\mu$, such that $\|\Phi_\mu\|_\mu\leq 2c\|\mathcal E\|_\mu$. 
	\end{proof}
	\section{Ljapunov–Schmidt Reduction}\label{sec-ls}
	Here we introduce the reduced energy associated to the energy functional \eqref{Energia}, given by 
	\begin{equation}\label{reduced}
		\mathfrak E(\xi,d) = E\(W_{\xi,\delta}+\Phi\),
	\end{equation}
	with $\Phi$ as in Proposition \ref{solve-aux} and
	\begin{equation}\label{d-choice}
		\delta(\mu) = \frac{d}{\mu},
	\end{equation}
	for some $d>0$ that will be chosen in the proof of Theorem \ref{main-theorem}.
	
	\medskip
	
	\begin{proposition}\label{ptocritico}
		If $(d,\xi)\in \R_+\times \R^{n-1}$ is a critical point of the reduced energy \eqref{reduced}, then $W+\Phi_\mu$ is a critical point of $E$, and so a solution of \eqref{P}.
	\end{proposition}
	\begin{proof}
		Take $s=1,\ldots,n-1.$ Since $(d,\xi)$ is a critical point for $\mathfrak E$, then
		\begin{equation}\label{ptocrit1}
			\begin{split}
				0 &= \frac{\de}{\de\xi^s}\, \mathfrak E(d,\xi) = \bigg\langle W+\Phi_\mu-i^*_\Omega\left(-(W+\Phi_\mu)^{n+2\over n-2}\right)\\ &-\left.i^*_{\de\Omega}\left(\frac{n-2}{2}\(\frac{D}{\sqrt{n(n-1)}}(W+\Phi_\mu)^{n\over n-2}-H\:(W+\Phi_\mu)\)\right),\frac{\de(W+\Phi_\mu)}{\de\xi^s}\right\rangle_\mu.
			\end{split}
		\end{equation}
		Using equation \eqref{aux}, we can write
		\begin{equation}\label{ptocritico2}
			\begin{split}
				&W+\Phi_\mu-\iM\left(-(\W+\Phi_\mu)^{n+2\over n-2}\right) \\ &-\idM\left(\frac{n-2}{2}\(\frac{D}{\sqrt{n(n-1)}}(\W+\Phi_\mu)^{n\over n-2}-H(\W+\Phi_\mu)\)\right) = \sum_{j=1}^nc_{j}\,\Z_j.
			\end{split}
		\end{equation}
		The proof concludes if we show that $c_{j,i}\to0$ as $\e\to 0$ for every $j=1,\ldots,n$. Using \eqref{ptocritico2}, we can rewrite \eqref{ptocrit1} as follows
		\begin{equation}\label{ptocritico3}
			0 = \sum_{j=1}^nc_{j}\scal{\Z_j,\frac{\de\W}{\de\xi^s}+\frac{\Phi_\mu}{\de\xi^s}}_\mu = \sum_{j=1}^nc_{j}\scal{\Z_j,\frac{\de\W}{\de\xi^s}}_\mu-\scal{\frac{\Z_j}{\de\xi^s},\Phi_\mu}_\mu,
		\end{equation}
		where for the last identity we have used that $\Phi_\e\in \mathcal K^\perp$, so 
		\begin{equation*}
			0 = \frac{\de}{\de \xi^s}\scal{\Z_j,\Phi_\mu}_\mu = \scal{\Z_j,\frac{\de\Phi_\mu}{\de\xi^s}}+\scal{\frac{\de \Z_j}{\de\xi^s},\Phi_\mu}.
		\end{equation*}
		
		The following estimates can be easily verified by direct computation (see \cite[Lemma E.1]{cpv}):
		\begin{equation}\label{ptocritico4}
			\begin{split}
				\scal{\Z_j,\frac{\de \W}{\de \xi^s}}_\mu &= \delta^{j s}\frac{1}{\delta}\frac{4(n-1)}{n-2}\norm{\nabla \mathfrak J_j}^2_{L^2(\R^n_+)}+\bigo{\mu\delta}, \\ \norm{\frac{\de \Z_j}{\de \xi^s}}_\mu^2 &= \frac{1}{\delta^2}\frac{4(n-1)}{n-2}\norm{\nabla\frac{\de\,\mathfrak J_j}{\de x_s}}_{L^2(\R^n_+)}^2+\bigo{\mu}.
			\end{split}
		\end{equation}
		
		Then, by \eqref{ptocritico3} and \eqref{ptocritico4}:
		
		\begin{equation*}
			0 = c_s \norm{\nabla \,\mathfrak J_s}_{L^2(\R^n_+)}^2+\bigo{\mu\delta^2}.
		\end{equation*}
		This proves $c_{s}\to 0$ as $\mu\to +\infty$ for every $s=1,\ldots,n-1.$ By taking derivatives with respect to $d$ and reasoning as before, we can also prove that $c_n\to 0$ as $\mu\to 0$, finishing the proof.
	\end{proof}
	In the next proposition we compute the energy of the approximating solution $W$, which will be the key part of the reduced functional \eqref{reduced}.
	
	\begin{proposition}\label{energy-building-block} Assume $n\geq 4$. Then the following expansion holds
		\begin{equation}\label{eq-EW}
			\begin{split}
				E(W) &= \left\lbrace\begin{array}{ll}\mathbb{E}+\mathscr C_n(D)H(\xi)\delta + \frac12 \left(\int_{\R^n_+}U_1^2\right)\mu\delta^2 + \bigo{\delta^2}&\text{if}\hsp n\geq 5, \\[1.5ex] \mathbb{E}+\mathscr C_4(D)H(\xi)\delta + \alpha_4^2\,\mu\delta^2\abs{\log \delta} + \bigo{\delta^2\abs{\log \delta}}&\text{if}\hsp n=4,\end{array}\right.
			\end{split}
		\end{equation}
		where
		\begin{equation} \label{def-Cn}
			\begin{split}
				\mathscr C_n(D)&=-(n-1)(n-2)\alpha_n^2(D^2\beta^2_n(D)+\beta^4_n(D))-\alpha_n^\dst \frac{n-2}{4n}\beta^2_n(D)\\ &+(n-2)\sqrt{n-1\over n}\alpha_n^\dsh D^2\beta_n^2(D)+(n-1)\alpha_n^2\beta^0_{n-2}(D),
			\end{split}
		\end{equation}
		and $\mathbb E$ is a constant representing the energy of $U_1$ in $\R^n_+$.
	\end{proposition}
	\begin{proof}
		By definition \eqref{Energia},
		\begin{equation*}
			\begin{split}
				E(u) &= \underbrace{\frac{2(n-1)}{n-2}\int_\Omega \abs{\nabla W_\delta}^2}_{{E_1}} + \underbrace{\frac{\mu}{2}\int_\Omega {W_\delta}^2}_{{E_2}} + \underbrace{\frac{n-2}{2n}\int_\Omega {W_\delta}^\dst}_{{E_3}} \\&-\underbrace{\frac{(n-2)D}{\sqrt{n(n-1)}}\int_{\de\Omega}{W_\delta}^\dsh}_{{E_4}} +\underbrace{(n-1)\int_{\de\Omega}H{W_\delta}^2}_{{E_5}}
			\end{split}
		\end{equation*}
		Let us address (${E_3}$) first:
		\begin{equation*}
			\begin{split}
				\int_\Omega {W_\delta}^\dst = \int_{C(\rho)} {U_{\delta}}^\dst - \int_\Sigma {U_{\delta}}^\dst
			\end{split}
		\end{equation*}
		On one hand,
		\begin{equation*}
			\int_{C(\rho)} {U_{\delta}}^\dst = \int_{C\left({\rho \over \delta}\right)} {U_{1}(y)}^\dst dy = \int_{\R^n_+}U_1(y)^\dst dy + \bigo{\delta^n}
		\end{equation*}
		On the other hand,
		\begin{equation*}
			\begin{split}
				\int_\Sigma {U_{\delta}}^\dst = \alpha_n^\dst\int_{B'(\rho)}\left(\int_0^{\varphi(\bar x)} \frac{\delta^n}{(\abs{\bar x}^2+(x_n+\delta D)^2-\delta^2)^n}dx_n\right)d\bar x \\ = \alpha_n^\dst\int_{B'\left({\rho\over\delta}\right)}\left(\int_0^{\varphi(\delta\bar y)\over \delta} \frac{1}{(\abs{\bar y}^2+(y_n+D)^2-1)^n}dy_n\right)d\bar y
			\end{split}
		\end{equation*}
		By Taylor expansion and \eqref{expansion}:
		\begin{equation}\label{taylor}
			\int_0^{\varphi(\delta \bar y)\over \delta} \frac{1}{(\abs{\bar y}^2+(y_n+D)^2-1)^n}dy_n= \frac{\delta \sum k_iy_i^2}{(\abs{\bar y}^2+D^2-1)^n}+\bigo{\delta^2}
		\end{equation}
		Then, 
		\begin{equation*}
			\begin{split}
				\int_\Sigma {U_{\delta}}^\dst &= \alpha_n^\dst\int_{B'\left({\rho\over\delta}\right)}\frac{\delta \sum k_iy_i^2}{(\abs{\bar y}^2+D^2-1)^n}d\bar y +\bigo{\delta^2} \\ &= \delta \alpha_n^\dst\frac{H(0)}{2}\int_{B'\left({\rho\over\delta}\right)}\frac{\abs{\bar y}^2}{(\abs{\bar y}^2+D^2-1)^n}d\bar y +\bigo{\delta^2} \\ &= \delta \alpha_n^\dst\frac{H(0)}{2} \beta^2_n(D) + \bigo{\delta^2}
			\end{split}
		\end{equation*}
		Finally,
		\begin{equation}\label{E3}
			({E_3}) = \frac{n-2}{2n}\int_{\R^n_+} {U_1}^\dst - \delta \alpha_n^\dst\frac{n-2}{4n} \beta^2_n(D)H(0) + \bigo{\delta^2}
		\end{equation}
		As for ({$E_4$}),
		\begin{equation*}
			\begin{split}
				\int_{\de\Omega} {W_\delta}^\dsh &= \int_{C(\rho)\cap\de\Omega}{U_\delta}^\dsh = \int_{B'(\rho)}\frac{\alpha_n^\dsh\delta^{n-1}}{(\abs{\bar x}^2+(\varphi(\bar x)+\delta D)^2-\delta^2)^{n-1}}d\bar x \\ &= \int_{B'\left({\rho\over\delta}\right)}\frac{\alpha_n^\dsh}{(\abs{\bar y}^2+({\varphi(\delta\bar y)\over \delta}+D)^2-1)^{n-1}}d\bar y
			\end{split}
		\end{equation*}
		By Taylor expansion,
		\begin{equation*}
			\begin{split}
				\int_{B'\left({\rho\over\delta}\right)}\frac{\alpha_n^\dsh}{(\abs{\bar y}^2+({\varphi(\delta\bar y)\over \delta}+D)^2-1)^{n-1}}d\bar y = \int_{B'\left({\rho\over\delta}\right)}\frac{\alpha_n^\dsh}{(\abs{\bar y}^2+D^2-1)^{n-1}}d\bar y \\ - \delta (n-1)2D \alpha_n^\dsh \int_{B'\left({\rho\over\delta}\right)} \frac{\sum k_iy_i^2}{(\abs{\bar y}^2+D^2-1)^n} d\bar y +\bigo{\delta^2} \\ = \int_{\de \R^n_+}U_1(\bar y,0)^\dsh d\bar y- \delta (n-1)D\alpha_n^\dsh \beta^2_n(D)H(0)+\bigo{\delta^2}
			\end{split}
		\end{equation*}
		Hence,
		\begin{equation}\label{E4}
			({E_4}) = - \frac{(n-2)D}{\sqrt{n(n-1)}}\int_{\de \R^n_+}U_1^\dsh + \delta (n-2) \sqrt{n-1\over n}\alpha_n^\dsh D^2\beta_n^2(D)H(0) + \bigo{\delta^2}
		\end{equation}
		We continue studying ({$E_1$}). It is easy to see that
		\begin{equation*}
			\int_\Omega \abs{\nabla W_\delta}^2 = \int_{\R^n_+}\abs{\nabla U_1}^2 - \int_{\Sigma} \abs{\nabla U_\delta}^2 +\bigo{\delta^2}
		\end{equation*}
		Reasoning as before,
		\begin{equation*}
			\begin{split}
				\int_{\Sigma} \abs{\nabla U_\delta}^2 &= \alpha_n^2(n-2)^2 \int_{B'\left({\rho\over\delta}\right)} \int_0^{\varphi(\delta\bar y)\over \delta} \frac{\abs{\bar y}^2+(y_n+D)^2}{(\abs{\bar y}^2+(y_n+D)^2-1)^n}dy_nd\bar y \\ &= \alpha_n^2 (n-2)^2 \delta\int_{B'\left({\rho\over\delta}\right)}\frac{(\abs{\bar y}^2+D^2)\sum k_iy_i^2}{(\abs{\bar y}^2+D^2-1)^n}d\bar y + \bigo{\delta^2} \\&= \frac{\alpha_n^2 (n-2)^2}{2} \delta H(0) \int_{B'\left({\rho\over\delta}\right)}\frac{(\abs{\bar y}^2+D^2)\abs{\bar y}^2}{(\abs{\bar y}^2+D^2-1)^n}d\bar y + \bigo{\delta^2} \\ &= \frac{\alpha_n^2 (n-2)^2}{2} \delta H(0) \left(D^2\beta^2_n(D) + \int_{B'\left({\rho\over\delta}\right)}\frac{\abs{\bar y}^4}{(\abs{\bar y}^2+D^2-1)^n}d\bar y\right) + \bigo{\delta^2} \\ &= \left\lbrace \begin{array}{ll}
					\dfrac{1}{2}\alpha_n^2 (n-2)^2  H(0) (D^2\beta^2_n(D) +\beta^4_n(D))\delta+\bigo{\delta^2} & \text{if} \hsp n\geq 4 \\[1.5ex] \bigo{\delta\abs{\log \delta}} & \text{if}\hsp n=3
				\end{array}\right.
			\end{split} 
		\end{equation*}
		Briefly,
		\begin{equation}\label{E1}
			({E_1}) = \int_{\R^n_+}\abs{\nabla U_1}^2 - (n-1)(n-2)\alpha_n^2 H(0) (D^2\beta^2_n(D)+\beta^4_n(D))\delta + \bigo{\delta^2}
		\end{equation}
		for $n\geq 4$.
		
		\medskip 
		
		({$E_5$}) can be addressed similarly to \eqref{E4}.
		\begin{equation*}
			\begin{split}
				\int_{\de \Omega} H(x){W_{\delta}}^2 &= \delta \int_{B'\left({\rho\over\delta}\right)}H(\delta \bar y) \frac{\alpha_n^2d\bar y}{(\abs{\bar y}^2+(\frac{\varphi(\delta \bar y)}{\delta}+D)^2-1)^{n-2}} \\ &= \alpha_n^2 H(0) \delta \left(\int_{B'\left({\rho\over\delta}\right)}\frac{d\bar y}{(\abs{\bar y}^2+D^2-1)^{n-2}}\right. \\ &\left.-\delta(n-2)H(0)D \int_{B'\left({\rho\over\delta}\right)}\frac{\abs{\bar y}^2d\bar y}{(\abs{\bar y}^2+D^2-1)^{n-1}}+\bigo{\delta^2}\right) \\ &= \left\lbrace\begin{array}{ll}  \alpha_n^2H(0)\beta^0_{n-2}(D)\delta +\bigo{\delta^2} & \text{if}\hsp n\geq 4 \\[1.5ex] \bigo{\delta \abs{\log \delta}} & \text{if}\hsp n=3 \end{array}\right.
			\end{split}
		\end{equation*}
		Then,
		\begin{equation}\label{E5}
			({E_5}) = (n-1)\alpha_n^2H(0)\beta^0_{n-2}(D)\delta +\bigo{\delta^2}
		\end{equation}
		for $n\geq 4$.
		
		\medskip
		
		Finally, let us consider $({E_2})$. 
		\begin{equation*}
			\int_{\Omega} W_\delta^2 = \int_{C(\rho)} U_\delta^2 - \int_\Sigma U_\delta^2 
		\end{equation*}
		We study the two integral terms separately. One one hand,
		\begin{equation*}
			\begin{split}
				\int_{C(\rho)} U_\delta^2 = \delta^2 \int_{C\left({\rho\over\delta}\right)} U_1(y)^2dy = \left\lbrace \begin{array}{ll} \delta^2 \int_{\R^n_+}U_1^2 & \text{if} \hsp n\geq 5 \\[1.5ex] \bigo{\delta^2 \abs{\log \delta}} & \text{if}\hsp n=4 \\[1.5ex] \bigo{\delta} & \text{if}\hsp n=3.\end{array}\right.
			\end{split}
		\end{equation*}
		On the other hand,
		\begin{equation*}
			\begin{split}
				\int_\Sigma U_\delta ^2 &= \delta^2 \int_{B'\left({\rho\over\delta}\right)} \int_0^{\varphi(\delta \bar y)\over \delta} \frac{\alpha_n^2}{(\abs{\bar y}^2+(y_n+D)^2-1)^{n-2}}dy_nd\bar y \\ &= {\alpha_n^2\over 2}  H(0) \delta^3 \int_{B'\left({\rho\over\delta}\right)} \frac{\abs{\bar y}^2}{(\abs{\bar y}^2+D^2-1)^{n-2}}d\bar y + \bigo{\delta^4} \\ &= \left\lbrace \begin{array}{ll} {\alpha_n^2\over 2}  H(0) \delta^3 \beta^2_{n-2}(D) & \text{if}\hsp n\geq 6 \\[1.5ex] \bigo{\delta^3\abs{\log \delta}}&\text{if}\hsp n=5 \\[1.5ex] \bigo{\delta^2}&\text{if}\hsp n=4 \\[1.5ex] \bigo{\delta}&\text{if}\hsp n=3\end{array}\right.
			\end{split}
		\end{equation*}
		Putting all together,
		\begin{equation}
			\begin{split}
				({E_2}) = \frac{1}{2}\mu\delta^2 \int_{\R^n_+} U_1^2 + \mu o(\delta^2)
			\end{split}
		\end{equation}
		for $n\geq 5$.
	\end{proof}
	\begin{proposition}\label{red1}
		Assume $n\geq 4$. It holds:
		\begin{equation*}
			\begin{split}
				E(W+\Phi_\mu) &= E(W)+\bigo{\norm{\Phi_\mu}_\mu^2}.
			\end{split}
		\end{equation*}
	\end{proposition}
	
	\begin{proof} To prove this result we need to repeat some estimates from the proof of Proposition \ref{error-size}, so many details will be skipped for the sake of brevity. 
		
		By Taylor expansion, there exists $\sigma \in (0,1)$ such that:
		\begin{align*}
			E(W+\Phi_\mu)-E(W) &= E'(W)[\Phi_\mu]+\frac{1}{2}\:E''(W+\sigma\: \Phi_\mu)[\Phi_\e,\Phi_\e]  \\ &=\frac{4(n-1)}{n-2}\int_\Omega \nabla W  \nabla\Phi_\mu + \mu\int_\Omega W\Phi_\mu+\int_\Omega  W^{n+2\over n-2}\Phi_\mu \\ &-\frac{2D(n-1)}{\sqrt{n(n-1)}}\int_{\de \Omega }W^{n\over n-2}\Phi_\mu+2 (n-1)\int_{\de \Omega} W\Phi_\mu \\ &+\frac{1}{2}\norm{\Phi_\mu}_\mu^2+(n-1)\int_{\de \Omega} \Phi_\mu^2 +\frac{n+2}{2(n-2)}\int_\Omega (W+\sigma\Phi_\mu)^{4\over n-2}\Phi_\mu^2 \\ &-\frac{2D\sqrt{n(n-1)}}{n-2}\int_{\de \Omega}( W+\sigma\Phi_\mu)^{2\over n-2}\Phi_\mu^2.
		\end{align*}
		By the Sobolev embeddings, we immediately have
		\begin{align*}
			\int_{\de \Omega} \Phi_\mu^2 &\leq  \norm{\Phi_\mu}^2_\mu, \\
			\int_\Omega (W+\sigma\Phi_\mu)^{4\over n-2}\Phi_\mu^2 &\leq C \|\Phi_\mu\|_\mu^2\:\| W+\sigma\Phi_\mu\|^{\frac{4}{n-2}}_{L^{\frac{2n}{n-2}}(\Omega)}\leq C \|\Phi_\mu\|_\mu^2, \\
			\int_{\de \Omega}( W+\sigma\Phi_\mu)^{2\over n-2}\Phi_\mu^2 &\leq C \|\Phi_\mu\|_\mu^2\:\| W+\sigma\Phi_\mu\|^{\frac{2}{n-2}}_{L^{\frac{2(n-1)}{n-2}}(\de \Omega)}\leq C \|\Phi_\mu\|^2_\mu.
		\end{align*}		
		On the other hand, integrating by parts,
		\begin{align*}
			&\frac{4(n-1)}{n-2}\int_\Omega \nabla W  \nabla\Phi_\mu + \mu\int_\Omega W\Phi_\mu+\int_\Omega  W^{n+2\over n-2}\Phi_\mu \\ &-\frac{2D(n-1)}{\sqrt{n(n-1)}}\int_{\de \Omega }W^{n\over n-2}\Phi_\mu+2 (n-1)\int_{\de \Omega} W\Phi_\mu
			\\ &= \int_\Omega\(-\frac{4(n-1)}{n-2}\Delta W +W^{n+2\over n-2}\)\Phi_\mu + \mu \int_\Omega W\Phi_\mu\\ &+2(n-1)\int_{\de \Omega} \(\frac{2}{n-2}\frac{\de W}{\de \eta} - \frac{D}{\sqrt{n(n-1)}}W^{n\over n-2}\)\Phi_\mu + \int_{\de \Omega}H W\Phi_\mu \\ &\leq C\norm{\Phi_\mu}_\mu \Bigg(\norm{-\frac{4(n-1)}{n-2}\Delta W+W^{n+2\over n-2}}_{L^{2n\over n+2}(\Omega)}+\mu  \norm{W}_{L^{2n\over n+2}(\Omega)} \\ &+\norm{\frac{2}{n-2}\frac{\de W}{\de \eta} - \frac{D}{\sqrt{n(n-1)}}W^{n\over n-2}}_{L^{2(n-1)\over n}(\de\Omega)}+\norm{W}_{L^{2(n-1)\over n}(\de \Omega)} \Bigg) \\ &=\bigo{\norm{\Phi_\mu}_\mu^2},
		\end{align*}
		and the claim follows.
	\end{proof}
	
	\section{Proof of Theorem \ref{main-theorem}}\label{sec-proof}
	According to Proposition \ref{ptocritico}, a critical point of the reduced energy \eqref{reduced} yields a corresponding critical point of $E$ and thus a solution to \eqref{P}. Using Propositions \ref{solve-aux} and \ref{energy-building-block}, along with the definition \eqref{d-choice}, we obtain the subsequent expansion for $n\geq 6$:
	\begin{equation*}
		\mathfrak E(d,\xi) = \mathbb{E}+\frac1\mu\(\mathscr C_n(D)H(\xi) d + \frac{d^2}{2} \int_{\R^n_+}U_1^2\) + \bigo{\frac1{\mu^2}},
	\end{equation*}
	where $\mathcal C_n(D)$ and $\mathbb{E}$ are defined on \eqref{def-Cn}. At first order, we have
	\begin{equation}\label{gradient}
		\nabla \mathfrak E(d,\xi) = \frac1\mu\(\mathscr{C}_n(D)H(\xi)+d\int_{\R^n_+}U_1^2, \mathscr C_n(D)d\:\nabla H(\xi)\).
	\end{equation}
	Let $p\in \de \Omega$ be such that $\nabla H(p)= 0$ and $H(p)>0$. Then, in view of \eqref{gradient}, 
	\begin{equation*}
		\(\frac{-\mathscr C_n(D)H(p)}{\norm{U_1}^2_{L^2(\R^n_+)}}, p\)
	\end{equation*}
	is a critical point of $\mathfrak E$ which is stable under $C^0$ perturbations. The proof of the Theorem concludes by showing that $\mathscr C_n(D)<0$ whenever $D>\sqrt{\frac{n+1}{n-1}}$, in order to have $d>0$. This fact follows directly from the lemma below, the proof of which involves straightforward calculations executed using mathematical software.
	\begin{lemma} Let $\mathscr C_n:(1,+\infty)\to\R$ be the function defined in \eqref{def-Cn}. The following statements are valid for every $n\geq 6.$
		\begin{enumerate}
			\item $\mathscr C_n(D)= \dfrac{\mathfrak a_n}{(D-1)^{\frac n2}}+\bigo{(D-1)^{\frac12}}$ as $D\to 1^+$, with $\mathfrak a_n >0$,
			
			\vspace{0.15cm}
			\item $\mathscr C_n\(D\)=0$  if and only if  $D=\sqrt{\dfrac{n+1}{n-1}}$,  and
			
			\vspace{0.15cm}
			\item  $\mathscr C_n(D)=-\dfrac{\mathfrak b_n D^3}{(D^2-1)^\frac{n}{2}} + \bigo{\dfrac{D}{(D^2-1)^{\frac n2}}}$ as $D\to +\infty$, with $\mathfrak b_n>0.$
		\end{enumerate}
	\end{lemma}

	\appendix
	
	\section{Proof of Lemma \ref{invertibility}}\label{app-invertibility}
	Suppose, by \textit{reductio ad absurdum}, that there exist sequences $\mu_m\to +\infty$, $d_m\to d>0$ and $\phi_m\in \mathcal K^\perp$ such that 
	\begin{equation*}
		\mathcal L(\phi_m)=\psi_m, \text{ with } \norm{\phi_m}_\mu=1 \text{ and } \norm{\psi_m}_\mu \to 0.
	\end{equation*}
	Let us denote $W_m=W(\mu_m,d_m)$, as defined in \eqref{A}. For convenience, we set
	\begin{equation*}
		\bar \phi_m(y) = \delta_m^{n-2\over 2}\phi_m(\delta_my)\,\chi(\delta_m y).
	\end{equation*}
	The fact that $\norm{\phi_m}_\mu = 1$ implies that $\bar\phi_m$ are bounded in $\mathcal D^{1,2}(\R^n_+)$, as it can be deduced from the following inequalities:
	\begin{equation*}
		\begin{split}
			\int_\Omega \abs{\nabla \phi_m}^2 &\geq \int_{\frac1\delta\Omega\cap C\({\rho\over\delta}\)}\abs{\nabla \bar\phi_m}^2, \\
			\int_\Omega \abs{\phi_m}^\dst & \geq \int_{\frac1\delta\Omega\cap C\({\rho\over\delta}\)}\abs{\bar \phi_m}^\dst,
		\end{split}
	\end{equation*}
	and the fact that $\frac1\delta\Omega\cap C\({\rho\over\delta}\)$ goes to the whole half-space $\R^n_+$ as $\delta\to 0.$ It follows that $\bar \phi_m\weakto \bar \phi$ weakly in $\mathcal D^{1,2}(\R^n_+)$ and $L^\dst(\R^n_+)$ and $\bar\phi_m\to \bar \phi$ strongly in $L^\dst_{loc}(\R^n_+)$. By the definition of $\mathcal L$, we are able to write
	\begin{equation}\label{inv-1}
		\psi_m - \phi_m - i^*_\Omega\(\frac{n+2}{n-2}W^{4\over n-2}\phi_m\)+i^*_{\de\Omega}\(\frac{n}{2}\frac{D}{\sqrt{n(n-1)}}W^{2\over n-2}\phi_m-H\phi_m\) = \sum_{i=1}^nC^i_m\mathcal Z_i,
	\end{equation}
	for some coefficients $C^i_m\in \R$. We will show that $C^i_m\to 0$ as $m\to +\infty$ for every $i=1,\ldots,n$. Consider the scalar product in $H^1_\mu(\Omega)$ of \eqref{inv-1} and $\mathcal Z_q$ to obtain
	\begin{equation*}
		\begin{split}
			\sum_{i=1}^nC^i_m\scal{\mathcal Z_i,\mathcal Z_q}_\mu &= \scal{i^*_\Omega\(-\frac{n+2}{n-2}W^{4\over n-2}\phi_m\),\mathcal Z_q}_\mu \\ &+ \scal{i^*_{\de\Omega}\(\frac{n}{2}\frac{D}{\sqrt{n(n-1)}}W^{2\over n-2}\phi_m-H\phi_m\)}_\mu.
		\end{split}
	\end{equation*}
	Integrating by parts, one can easily show that $\scal{i^*_\Omega(\mathfrak f),\mathcal Z_q}_\mu = \scal{\mathfrak f,\mathcal Z_q}_{L^2(\Omega)}$ \newline and $\scal{i^*_{\de \Omega}(\mathfrak g),\mathcal Z_q}_\mu=-\frac{4(n-1)}{n-2}\scal{\mathfrak g, \mathcal Z_q}_{L^2(\de\Omega)}$. Therefore,
	\begin{equation}\label{inver-eq-0}
		\begin{split}
			\sum_{i=1}^nC^i_m\scal{\mathcal Z_i,\mathcal Z_q}_\mu &= -\frac{n+2}{n-2}\int_\Omega W^{4\over n-2}\phi_m \mathcal Z_q  - \frac{2D\sqrt{n(n-1)}}{n-2}\int_{\de\Omega}W^{2\over n-2}\phi_m \mathcal Z_q \\&+\frac{4(n-1)}{n-2} \int_{\de \Omega} H\phi_m \mathcal Z_q.
		\end{split}
	\end{equation}
	Using the orthogonality of $\mathfrak J_i$ in $H^1(\R^n_+)$, one can easily deduce that 
	\begin{equation}\label{scal-prod-J}
		\scal{\mathcal Z_i,\mathcal Z_q}_\mu = \delta^{iq}\norm{\mathfrak J_q}^2_{H^1_\mu(\R^n_+)} = \delta^{iq}\:\bigo{1+\mu_m\delta_m^2} \hsp \text{for } n\geq 4.
	\end{equation}
	Next, we estimate the last term in the right-hand side of \eqref{inver-eq-0} as follows
	\begin{equation}\label{inver-eq-1}
		\begin{split}
			\int_{\de \Omega} H \phi_m \mathcal Z_q &= \frac{1}{\delta^{n-2\over 2}}\int_{\de\Omega\cap C(0,\rho)}H(x)\phi_m(x)\mathfrak J_q\({x\over \delta_m}\)dx \\ &\simeq H(0)\delta_m\int_{\frac1{\delta_m}B'_+(0,\rho)}\bar\phi_m(\bar x,\phi(\bar x)) \:\mathfrak J_q\(\bar x,\frac{\phi(\delta_m\bar x)}{\delta_m}\) d\bar x = o\(1\).
		\end{split}
	\end{equation}
	Integrating by parts and using \eqref{L}, one can notice that
	\begin{equation}\label{inver-eq-2}
		\begin{split}
			0 = \scal{\phi_m,\mathcal Z_q}_{\mu_m} &= \frac{-4(n-1)}{n-2}\int_\Omega \Delta \mathcal Z_q \phi_m + \mu \int_\Omega \phi_m \mathcal Z_q + \frac{4(n-1)}{n-2}\int_{\de \Omega}\frac{\de \mathcal Z_q}{\de \nu}\phi_m \\ &= -\frac{n+2}{n-2}\int_{\frac1{\delta_m}C(0,\rho)} W^{4\over n-2}\mathcal J_q \bar \phi_m +\bigo{\mu\delta^2} \\ &+ \frac{2D\sqrt{n(n-1)}}{n-2}\int_{\frac1{\delta_m}B'(0,\rho)}W^{2\over n-2}\mathfrak J_q\bar \phi_m
		\end{split}
	\end{equation}
	Using equations \eqref{scal-prod-J}, \eqref{inver-eq-1} and \eqref{inver-eq-2} in \eqref{inver-eq-0} and rescaling as before, we obtain
	\begin{equation*}
		C^q_m \to 0, \hsp \text{for every } 1\leq q \leq n,
	\end{equation*}
	as desired. Now, fix any $\varphi\in C^2(\R^n_+)$ with compact support and a cut-off function $\chi$ and define
	\begin{equation*}
		\varphi_m(x)=\frac{1}{\delta^{n-2\over 2}}\varphi\(\frac x\delta\)\chi(x).
	\end{equation*}
	We multiply \eqref{inv-1} by $\varphi_m$ to obtain
	\begin{equation}\label{inv-eq-3}
		\begin{split}
			o_m(1)&=\scal{\psi_m,\varphi_m}_{\mu_m} - \scal{\phi_m,\varphi_m}_{\mu_m} + \frac{n+2}{n-2}\int_\Omega W^{4\over n-2}\phi_m \varphi_m \\ &-\frac{2D\sqrt{n(n-1)}}{n-2}\int_{\de \Omega} W^{2\over n-2} \phi_m\varphi_m + \frac{4(n-1)}{n-2}\int_{\de M}H\phi_m \varphi_m \\ &= \frac{4(n-1)}{n-2} \int_{\frac1\delta C(0,\rho)} \nabla \bar\phi_m \nabla \varphi + \frac{n+2}{n-2}\int_{\frac1\delta C(0,\rho)}U^{4\over n-2}\bar \phi_m\varphi \\ &-\frac{2D\sqrt{n(n-1)}}{n-2}\int_{\frac1\delta B'(0,\rho)} U^{2\over n-2} \bar \phi_m \varphi +o_m(1).
		\end{split}
	\end{equation}
	Since $\varphi$ was arbitrary, passing to the limit in \eqref{inv-eq-3} we find that $\bar \phi$ is a weak solution of \eqref{L}. Using \cite[Th. 2.1]{cpv}, we have $\bar\phi \in \text{span}\(\mathfrak J_i:\:i=1,\ldots,n\)$. However, the orthogonality of $\phi_m$ with respect to every $\mathcal Z_i$ in $H^1(\Omega)$ implies after rescaling that $\bar\phi=0$. Then, multiplying \eqref{inv-1} by $\phi_m$ and proceeding as before, we see that
	\begin{equation*}
		\begin{split}
			\norm{\phi_m}_{\mu_m}^2 &= \frac{n+2}{n-2}\int_{\frac1\delta C(0,\rho)}U^{4\over n-2}\bar {\phi_m}^2 -\frac{2D\sqrt{n(n-1)}}{n-2}\int_{\frac1\delta B'(0,\rho)} U^{2\over n-2} \bar {\phi_m}^2 +o_m(1) \\ &= o_m(1).
		\end{split}
	\end{equation*}
	This contradicts the assumption $\norm{\phi_m}_{\mu_m} = 1$.
	\section{Proof of Lemma \ref{contraction}}\label{app-contraction} First of all, let us put $F(t)=t^{n+2\over n-2}$ and $G(t)= \frac{n-2}{2}t^{n\over n-2}$ to simplify the notation. By the continuity of $\i^*_\Omega$ and $\i^*_{\de \Omega}$, we have
	\begin{align*}
		\norm{N(\phi_1)-N(\phi_2)}_\mu \leq \norm{F(W+\phi_2)-F(W+\phi_1)+F'(W)[\phi_1-\phi_2]}_{L^\frac{2n}{n+2}(\Omega)}& \\ + \frac{D}{\sqrt{n(n-1)}}\norm{G(W+\phi_2)-G(W+\phi_1)+G'(\W)[\phi_1-\phi_2]}_{L^\frac{2(n-1)}{n}(\de \Omega)}& 
	\end{align*}
	Expanding $F(W+\phi_2)$ and $G(W+\phi_2)$ around $W+\phi_1,$ we obtain $0<\alpha,\beta < 1$ such that 
	\begin{equation}\label{contr:cont}
		\begin{split}
			\norm{\mathcal N(\phi_1)-\mathcal N(\phi_2)}_\mu \leq \norm{\left(F'(W)-F'(W+\alpha\phi_1+(1-\alpha)\phi_2)\right)[\phi_1-\phi_2]}_{L^\frac{2n}{n+2}(\Omega)}& \\ +\frac{D}{\sqrt{n(n-1)}}\norm{\left(G'(W)-G'(W+\beta\phi_1+(1-\beta)\phi_2)\right)[\phi_1-\phi_2]}_{L^\frac{2(n-1)}{n}(\de \Omega)}.
		\end{split}
	\end{equation}
	We remind a well-known inequality: for every $a,b\in\R$ and $q>0$, it holds
	\begin{equation}\label{contr:ineq}
		\abs{\abs{a+b}^q-a^q}\leq c(q)\times\left\lbrace\begin{array}{ll} \min\{\abs{b}^q,a^{q-1}\abs{b}\} & \text{if}\hsp q<1, \\ \left(\abs{a}^{q-1}\abs{b}+\abs{b}^q\right) & \text{if}\hsp q\geq 1.
		\end{array} \right.
	\end{equation}
	By \eqref{contr:ineq}, we have 
	\begin{equation}\label{contr:boundm}
		\begin{split}
			&\abs{\abs{W}^\frac{4}{n-2}-\abs{W+\alpha\phi_1+(1-\lambda)\phi_2}^\frac{4}{n-2}}\\&\leq c(n)\times\left\lbrace\begin{array}{ll} \abs{\alpha\phi_1+(1-\alpha)\phi_2}^\frac{4}{n-2} & \text{if}\hsp n\geq6, \\ \abs{\alpha\phi_1+(1-\alpha)\phi_2}^\frac{4}{n-2}+\abs{W}^\frac{6-n}{n-2}\abs{\alpha\phi_1+(1-\alpha)\phi_2} & \text{if}\hsp n=4,5,
			\end{array} \right.
		\end{split}
	\end{equation}
	and
	\begin{equation}\label{contr:bounddm} 
		\abs{\abs{W}^\frac{2}{n-2}-\abs{\W+\beta\phi_1+(1-\beta)\phi_2}^\frac{2}{n-2}} \leq \abs{\beta\phi_1+(1-\beta)\phi_2}^\frac{2}{n-2}.
	\end{equation}
	On one hand, given the fact that $\phi_i\in L^\frac{n+2}{n-2}(\Omega)$, by \eqref{contr:boundm} and Hölder's inequality:
	\begin{equation}\label{c:int}
		\begin{split}
			&\norm{\left(F'(W)-F'(W+\alpha\phi_1+(1-\alpha)\phi_2)\right)[\phi_1-\phi_2]}_{L^\frac{2n}{n+2}(\Omega)} \leq C\norm{\phi_1-\phi_2}_\mu\\&\times \left\lbrace\begin{array}{ll} \norm{\alpha\phi_1+(1-\alpha)\phi_2}_{L^\frac{2n}{n-2}(\Omega)}^\frac{4}{n-2}&\text{if}\hsp n\geq 6, \\\norm{\alpha\phi_1+(1-\alpha)\phi_2}_{L^\frac{2n}{n-2}(\Omega)}^\frac{4}{n-2}+ \norm{W}_{L^\frac{2n}{n-2}(\Omega)}^\frac{6-n}{n-2}\norm{\alpha\phi_1+(1-\alpha)\phi_2}_{L^\frac{2n}{n-2}(\Omega)} & \text{if}\hsp n=4,5. \end{array}\right.
		\end{split}
	\end{equation}
	On the other hand, since ${\phi_i}^\frac{2(n-1)}{n}\in L^\frac{n}{n-2}(\de \Omega)$, by \eqref{contr:bounddm} we obtain
	\begin{equation}\label{c:boundary}
		\begin{split}
			&\norm{\left(G'(W)-G'(W+\beta\phi_1+(1-\beta)\phi_2)\right)[\phi_1-\phi_2]}_{L^\frac{2(n-1)}{n}(\de \Omega)} \\ &\leq C\norm{\phi_1-\phi_2}_\mu\norm{\beta\phi_1+(1-\beta)\phi_2}_{L^\frac{2(n-1)}{n-2}(\de \Omega)}^\frac{2}{n-2}.
		\end{split}
	\end{equation}
	In view of \eqref{c:int} and \eqref{c:boundary},
	\begin{equation*}
		\norm{\mathcal N(\phi_1)-\mathcal N(\phi_2)}_{H^1(M)}\leq \gamma \norm{\phi_1-\phi_2}_\mu,\hsp 0<\gamma<1,
	\end{equation*}
	provided $\phi_1$ and $\phi_2$ are small enough.

\end{document}